\documentclass[12pt]{article}
\usepackage[a4paper, total={6in, 10in}]{geometry}
\usepackage{graphicx} % Required for inserting images
\usepackage{soul}

\usepackage{amsmath, amsthm, amssymb, amsfonts, mathtools, mathrsfs}

\usepackage[utf8]{inputenc}
\usepackage{tikz-cd}
\usepackage[english]{babel}
\usepackage{etoolbox}
\usepackage{comment}
\usepackage{enumerate}
\usepackage{color}
\usepackage{bbm}
\usepackage{setspace}
\usepackage{bbm}
\usepackage[colorlinks]{hyperref}
\usepackage{float}
\usepackage{dsfont}
\usepackage{cancel}
\usepackage[nice]{nicefrac}
\usepackage{xfrac}
\usepackage[
backend=biber,
style=alphabetic,
sorting=ynt,
maxbibnames=99
]{biblatex}
\newcommand{\lora}{\longrightarrow}

\usepackage{xcolor}
\definecolor{Aranka}{RGB}{0, 135, 85}

\definecolor{Tomer}{RGB}{200, 20, 0}

\newtheorem{theorem}{Theorem}[section]
\newtheorem{lemma}[theorem]{Lemma}
\newtheorem{proposition}[theorem]{Proposition}
\newtheorem{corollary}[theorem]{Corollary}

\newtheorem{definition}[theorem]{Definition}
\theoremstyle{definition}

\newtheorem{remark}[theorem]{Remark}
\newtheorem{example}[theorem]{Example}

\newtheorem{conjecture}[theorem]{Conjecture}

\newcommand{\Pcal}{\mathcal{P}}
\newcommand{\Acal}{\mathcal{A}}

\newcommand{\Ecal}{{\mathcal E}}
\newcommand{\Rcal}{{\mathcal R}}
\newcommand{\CC}{\mathbb{C}}
\newcommand{\ZZ}{\mathbb{Z}}
\newcommand{\FF}{\mathbb{F}}

\newcommand\C{\mathbb C}

\newcommand\N{\mathbb N}
\newcommand\Z{\mathbb Z}

\newcommand{\cay}{\text{Cay}}

\newcommand\G{\Gamma}

\newcommand{\epsi}{\varepsilon}
\newcommand{\goes}{\longrightarrow}
\newcommand{\f}{\varphi}
\newcommand{\g}{\gamma}

\newcommand{\spec}{\text{spec}}
\newcommand{\Prob}{\text{Prob}}
\newcommand{\pr}{{\rm Prob}}
\newcommand{\Image}{\text{Im}}
\newcommand{\goesn}{\underset{n\to\infty}{\goes}}

\newcommand{\ler}[1]{\left( #1 \right)} % left, right simple brackets
\newcommand{\lers}[1]{\left\{ #1 \right\}} % curly brackets
\newcommand{\lesq}[1]{\left[ #1 \right]}% square brackets
\newcommand{\abs}[1]{\left| #1 \right|} % absolute value
\newcommand{\innprod}[1]{\left< #1 \right>} %inner product

\newcommand{\acts}{\curvearrowright}
\newcommand{\act}{\curvearrowright}

\title{Retraction Theorems for Group Compactifications}
\usepackage{bigfoot}

\DeclareNewFootnote{AAffil}[arabic]
\DeclareNewFootnote{ANote}[fnsymbol]

\usepackage{etoolbox}
\makeatletter
\patchcmd\maketitle{\def\@makefnmark{\rlap{\@textsuperscript{\normalfont\@thefnmark}}}}{}{}{}
\makeatother

\makeatletter
\def\thanksAAffil#1{% <--- These %'s are necessary for spacing
  \footnotemarkAAffil\protected@xdef\@thanks{\@thanks%
        \protect\footnotetextAAffil[\the \c@footnoteAAffil]{#1}}%
}
\def\thanksANote#1{%
  \footnotemarkANote%
  \protected@xdef\@thanks{\@thanks%
        \protect\footnotetextANote[\the \c@footnoteANote]{#1}}%
}
\makeatother

\author{% <---- Not sure if these %'s are necessary, but can't hurt
  Yair Hartman%
  \thanksAAffil{Ben-Gurion University of the Negev; email: hartmany@bgu.ac.il}%
  , %
  Aranka Hrušková%
  \thanksAAffil{The Weizmann Institute of Science; email: umim.cist@gmail.com}%
  , %
  Mehrdad Kalantar%
  \thanksAAffil{University of Houston; email: mkalantar@uh.edu}%
  , %
Tomer Zimhoni%
  \thanksAAffil{Ben-Gurion University of the Negev; email: zimhoni@post.bgu.ac.il}%
}

\date{}%{September 2025}

\begin{document}

\maketitle

\begin{abstract}
We characterize group compactifications of discrete groups for which there exists an equivariant retraction onto the boundary. In particular, we prove an equivariant analogue of Brouwer's No-Retraction theorem for large classes of group compactifications, which includes actions of hyperbolic groups on their Gromov boundary.
\end{abstract}

\section{Introduction}

One of the original inspirations for the modern theory of boundary actions of topological groups was the case of the compact unit disc $\mathbb{D}$, where the harmonic functions are in correspondence with functions on the boundary $\partial\mathbb{D}$.
This inspired Furstenberg \cite{Fur63,Fur73} to introduce far-reaching generalizations of compactifications of topological groups, whose resulting boundary theories have proven to be a powerful tool in the study of non-amenable groups and various structures corresponding to them, most prominently in rigidity theory of semisimple groups and their lattices \cite{Fur73,margulis1978quotient,stuck1994stabilizers,bader2022charmenability}.

Returning to the classical case of the compact disc $\mathbb{D}$, one of the most celebrated results in the context is Brouwer's No-Retraction theorem, which states that there is no retraction from $\mathbb{D}$ onto its boundary $\partial\mathbb{D}$.

The main goal of this paper is to investigate analogues of Brouwer's theorem in the context of group compactifications considered by Furstenberg \cite{Fur71}.

Given a countable discrete group $\G$, by a $\G$-compactification we mean a compactification $\overline{\G}$ such that the action $\G\act \G$ by left multiplication extends to a continuous action $\G\act \overline{\G}$. We focus particularly on the case where the resulting action $\G\act \partial\overline{\G}$ of $\G$ on the boundary $\partial\overline{\G}:=\overline{\G}\backslash \G$ is minimal, that is, has no proper closed $\G$-invariant subspace.

Among the most natural examples of this type of compactifications is the Gromov compactification of the free group $\FF_n$ (or more generally, hyperbolic groups). We prove an equivariant analogue of Brouwer's No-Retraction theorem in this case.

\begin{theorem} \label{Brouwer for the gromov comp.}
Let $\Gamma$ be a non-elementary hyperbolic group, and denote by $\partial\Gamma$ its Gromov boundary and by $\overline{\G}$ the Gromov compactification of $\G$. Then there exists {\bf no} $\Gamma$-equivariant retraction $\overline{\Gamma} \to \partial\Gamma$.
\end{theorem}

Some remarks are in order. First, we note that the condition of equivariance is necessary in this context. Indeed, some retraction $\Gamma\cup\partial\Gamma \to \partial\Gamma$ always exists owing to $\Gamma$ being discrete.
Second, we will see that the ``no-retraction'' property for a $\G$-compactification is essentially related to the manner of ``gluing'' of the group $\G$ to the boundary of the compactification. Concretely, we will see that there is an $\FF_2$-compactification $\overline{\FF_2}$ whose boundary $\partial\overline{\FF_2}$ is the Gromov boundary $\partial \FF_2$, and such that there exists a retraction $\overline{\FF_2}\to \partial\overline{\FF_2}$.

Let us recall the construction of the $\G$-compactifications studied by Furstenberg in \cite{Fur71}.

Let $X$ be a compact metrizable minimal $\G$-space, and let $\nu\in\pr(X)$ be a contractive measure (i.e., the weak* closure of the orbit $\{\g\nu : \g\in \G\}$ contains Dirac measures). Define the topology $\tau_\nu$ on the disjoint union $\G\cup X$ to be the one that on $X$ restricts to the original topology on $X$, on $\G$ to the discrete topology on $\G$, and such that for an unbounded sequence $(\g_n)$ in $\G$ and $x\in X$,  $\g_n \to x$ if and only if $\g_n\nu \to \delta_x$ in the weak* topology on $\pr(X)$.
In Lemma~\ref{lem: measured compactification} we show that there indeed exists a unique topology on $\G\cup X$ that satisfies the above properties, and it does define a $\G$-compactification. We refer to this class of $\G$-compactifications as \emph{orbital} compactifications, and if $\nu=\delta_{x_0}$ is a Dirac measure at some point $x_0\in X$, we call it a \emph{point-orbital} compactification.

Examples \ref{example: two points compactification, non-trivial} and \ref{example: infinite dihedral grp} show that not all $\G$-compactifications are orbital, however, we prove the following.

\begin{theorem} \label{non-Dirac iff Brouwer}
Let $\overline{\G}$ be a $\G$-compactification with minimal boundary $\partial \overline{\G}$. Then there exists a retraction $\overline{\G}\to\partial \overline{\G}$ if and only if $\overline{\G}$ is point-orbital.
\end{theorem}

    \subsection*{Acknowledgments}
    YH, AH and TZ were supported by the funding from the European Research Council (ERC) under the European Union's Seventh Framework Programme (FP7-2007-2013) (Grant agreement
No. 101078193). MK was supported by the NSF Grant DMS-2155162.
    \section{Preliminaries}\label{section: Preliminaries}

        Let $\Gamma$ denote a countable discrete group. We denote by $\ell^\infty(\Gamma)$ the algebra of all complex-valued bounded functions from $\Gamma$ and by 
        $c_0(\G)$ the ideal of all functions ``vanishing at infinity'', namely, the space of all those functions $f\in \ell^\infty(\G)$ such that for every $\epsi>0$, the set $\{\gamma\in \G : \abs{f\ler{\gamma}}>\epsi\}$ is finite.
        
An action $\Gamma\acts X$ of $\G$ on a compact Hausdorff space $X$ by homeomorphisms
induces further canonical actions of $\G$. First, the action on the compact convex space $\pr(X)$ of all Borel probability measures on $X$ by affine homeomorphisms given by
$(g \cdot \nu)(E)=\nu(g^{-1}E)$ for $\nu\in\pr(X)$, $g\in\G$.
Second the action on $C(X)$ by *-automorphisms, defined by 
$(g \cdot f)(x)=f(gx)$ for $g\in\G$, $f\in C(X)$.

Given a measure $\nu\in \pr(X)$, the corresponding \emph{Poisson transform} $P_{\nu}\colon C(X)\lora \ell^{\infty}(\Gamma)$ is defined by $P_{\nu}(f)(\g)=\int_Xf\,d\g\nu$.
The transform $P_{\nu}$ is unital, positive and $\Gamma$-equivariant. 
Conversely, every unital positive $\Gamma$-equivariant map $\varphi\colon C\ler{X} \to \ell^\infty\ler{\Gamma}$ is a Poisson transform, namely, $\varphi=P_\nu$ for $\nu\in \pr(X)$ which is defined by $\nu(f)=[\varphi(f)](e)$ for every $f\in C\ler{X}$. 

The Poisson transform is multiplicative (hence a *-homomorphism) if and only if $\nu$ is a Dirac measure.

        By Gelfand duality, a unital $\G$-invariant $C^*$-subalgebra $\Acal$ of $\ell^\infty\ler{\Gamma}$ is of the form $C(Z)$ for some compact $\G$-space $Z$. Furthermore, if $c_0\ler{\Gamma} \subset \Acal$, then $\spec\ler{\Acal}$ contains a copy of $\Gamma$.

\subsection{Group compactifications}
As mentioned in the introduction, in this paper, we restrict our attention to those group compactifications whose boundary is minimal as a $\G$-space.

\begin{definition}[$\Gamma$-compactification] \label{def: Gamma-comp.}
Let $X$ be a compact minimal $\G$-space.
A topology $\tau$ on the set $\Gamma\cup X$ is called a {\bf $\Gamma$-compactification} if
\begin{enumerate}
\item  the space $(\Gamma\cup X,\tau)$ is compact and Hausdorff, $\Gamma$ is dense in $\Gamma\cup X$, and the induced topologies $\tau|_\Gamma$ and $\tau|_X$ are the original topologies on $\Gamma$ and $X$ respectively,
\item the actions $\Gamma\acts \Gamma$ (by left multiplication) and $\G\act X$ together define a continuous action of $\Gamma$ on the union $\Gamma\cup X$.
\end{enumerate}
\end{definition}

Our main focus will be on the case where $X$ is metrizable, and for the rest of the paper, \emph{unless otherwise stated, $X$ is assumed to be metrizable}. Note that in this case, any compactification $\Gamma\cup X$ is also metrizable since by countability of $\G$, one can choose a countable separating family of continuous functions on $\Gamma\cup X$.
Definition~\ref{def: Gamma-comp.} is, however, valid for general compact $\G$-spaces $X$, and indeed there are important examples of $\Gamma$-compactifications in which $X$ is non-metrizable (e.g. the Furstenberg boundary compactification).
The restriction to the metrizable case simplifies our constructions, however, all our results hold with modified arguments, replacing sequences with nets.

When a compact metrizable minimal $\G$-space $X$ is given, in order to define a concrete topology $\tau$ on the disjoint union $\G\cup X$ that turns it into a 
$\Gamma$-compactification, our task, informally speaking, 
 
is to describe the ``gluing of $\Gamma$ to $X$" in a ``legal" way: declaring which sequences in $\G$ converge to which points in $X$ in a way that 
(i) every sequence has a convergent subsequence (compactness);
(ii) if a sequence converges to some $x\in X$, then so do all of its subsequences (well-definedness); % topology),            \item (Continuous action) I
(iii) if $\gamma_n \underset{n\to\infty}{\goes}x$, then for every $g\in \Gamma$, $g\gamma_n \underset{n\to\infty}{\goes} gx$ (continuity of the $\Gamma$-action); and (iv) 
for every $x\in X$, there exists a sequence in $\Gamma$ which converges to it (density).

We present several such constructions in later sections.

\medskip

We often work in the dual setup of function algebras associated to $\G$-actions and their compactifications.
Let us record some basic facts that we will use later.
\begin{lemma}\label{a function vanishes on X iff in c0}
Let $\ler{\Gamma\cup X, \tau}$ be a $\Gamma$-compactification, and let $f\in C\ler{\Gamma\cup X}$. Then $f|_X \equiv0$ if and only if $f|_{\Gamma} \in c_0 \ler{\Gamma}$.
\end{lemma}

\begin{proof}
If $f|_\Gamma \in c_0\ler{\Gamma}$, then $f|_X \equiv0$ by density of $\Gamma$ in $\Gamma\cup X$, as for $x\in X$, there exists an injective sequence $(\gamma_n)$ in $\Gamma$ of which $x$ is an accumulation point. This sequence must satisfy that $f\ler{\gamma_n} \goesn  0$ by virtue of $f|_\G$ being in $c_0(\G)$, but then necessarily $f(x)=0$ by continuity of $f$.

Conversely, suppose that $f|_X \equiv0$, and for the sake of contradiction, suppose that $f|_\Gamma \notin c_0\ler{\Gamma}$. Then there exists $\epsi>0$ and an injective sequence $\ler{\gamma_n}_n$ in $\G$ such that $\abs{f\ler{\gamma_n}}>\epsi$ for every $n$. By compactness, $(\gamma_n)$ must have a convergent subsequence $\ler{\gamma_{n_k}}_k$, and by injectivity of $\ler{\gamma_n}$ and the fact that $\Gamma$ is discrete, $\ler{\gamma_{n_k}}_k$ converges to some $x\in X$. But then, by continuity of $f$, $f\ler{\gamma_{n_k}}\underset{k\to\infty}{\goes}  f\ler{x}=0$ in contradiction with that $\abs{f\ler{\gamma_{n_k}}}>\epsi$ for all $k$.
    \end{proof}

Denote by $i\colon c_0\ler{\Gamma}\to C\ler{\Gamma\cup X}$ the inclusion defined by $i\ler{f}|_\G=f$ and $i\ler{f}|_X\equiv0$.
Note that $i(f)$ is indeed in $C(\Gamma\cup X)$ as $\G$ is locally compact and hence open in any of its compactifications. Also denote by $\pi\colon C\ler{\Gamma\cup X} \to C\ler{X}$ the restriction $\pi\ler{f}=f|_X$.
Then we get that a $\Gamma$-compactification gives rise to a short exact sequence of commutative $\Gamma$-$C^*$-algebras as follows:
\begin{equation}\label{exact sequence}
\begin{tikzcd}
1 & {c_0 \left( \Gamma \right)} & {C\left(         \Gamma\cup X\right)} & {C\left(X\right)} & 1.
\arrow[from=1-1, to=1-2]
\arrow["i", from=1-2, to=1-3]\arrow["\pi", from=1-3, to=1-4]\arrow[from=1-4, to=1-5]
\end{tikzcd}   
\end{equation}

\begin{proposition} \label{prop: every compactification is a spec of a subalgebra}
For every $\Gamma$-compactification $\ler{\Gamma\cup X, \tau}$, there exists a (canonical) $\Gamma$-invariant $C^*$-subalgebra $\Acal_\tau \leq \ell^\infty \ler{\Gamma}$ such that ${\Acal_\tau} \cong C\ler{\Gamma \cup X , \tau}$.
\end{proposition}

\begin{proof}
Consider the restriction $*$-homomorphism $r \colon C\ler{\Gamma\cup X , \tau}$, $f \mapsto f|_\Gamma$,
and let $\Acal_\tau := \Image\ler{r}$. The fact that $\Acal_\tau$ is a $\Gamma$-invariant sub-$C^*$-algebra of $\ell^\infty \ler{\Gamma}$ is clear, and the fact that $r$ is an isomorphism to its image follows from the fact that $\Gamma$ is dense in $\ler{\Gamma \cup X , \tau}$ and hence $r$ must be injective. 
\end{proof}

\section{Orbital compactifications} \label{section: oribital compactification}
Our main interest in this paper is a class of $\Gamma$-compactifications whose topology is defined by the convergence along orbits of measures on the boundary of the compactification. This construction was presented by Furstenberg in \cite{Fur71} in the context of boundary actions.

\begin{lemma}
\label{lem: measured compactification}
Let $X$ be a compact metrizable $\Gamma$-space, and let $\nu\in\pr(X)$. Then there is a unique topology $\tau_\nu$ on the union $\Gamma\cup X$ satisfying that
\begin{enumerate}
\item[{\rm (i)}\,]
the canonical inclusions of $\G$ and $X$ into $\Gamma\cup X$ are homeomorphisms onto their images, and
\item [{\rm (ii)}]
for every sequence $(\gamma_n)$ in $\G$ which leaves every finite set of $\G$ and every $x\in X$, $\gamma_n \goesn  x$ if and only if $\gamma_n\nu \goesn  \delta_x$ weak*.
\end{enumerate}
\end{lemma} 

\begin{proof}
We define the topology $\tau_{\nu}$ through the closed sets. We declare that $S\subset \Gamma \cup X$ is closed if and only if
\begin{enumerate}
\item[A.] $S\cap \Gamma$ and $S\cap X$ are closed in $\Gamma$ and in $X$ respectively, and
\item[B.] whenever a sequence $(\gamma_n)$ in $S\cap\Gamma$ leaves every finite set of $\G$, the set $S\cap X$ contains
\[\{x\in X : \delta_x \text{ is a weak* accumulation point of the sequence } (\gamma_n\nu)\}.\]
\end{enumerate}
The collection of sets $S$ which satisfy these two conditions is clearly closed under arbitrary intersections, and contains $\emptyset$ and $\Gamma\cup X$. To see that it is closed under finite unions, let $S_1,S_2$ be two sets that satisfy our two conditions. Then for $S_1 \cup S_2$, condition A holds evidently. To see that B holds, let $(\gamma_n)$ be a sequence in $(S_1\cup S_2)\cap\G$ which leaves every finite set of $\G$ and such that $\delta_x$ is an accumulation point of $(\gamma_n\nu)$ for some $x\in X$. Without loss of generality, suppose that infinitely many elements in $(\gamma_n)$ are from $S_1$, and so form a subsequence $(\gamma_{n_k})_k$ in $S_1\cap\G$ such that $\delta_x$ is an accumulation point of $\ler{\gamma_{n_k}\nu}_k$, and we conclude $x \in S_1$ and hence $x\in S_1\cup S_2$. It follows that the topology $\tau_\nu$ is well defined.
   
To show uniqueness, let $\tau$ be a topology on $\G\cup X$ which satisfies the conditions (i) and (ii) in the statement. Let $S$ be a set closed with respect to $\tau$. It is straightforward to show that it must satisfy conditions A and B and therefore it must be closed in $\tau_{\nu}$.
Conversely, assume $S$ is a closed set of $\tau_{\nu}$, and let $(s_n)$ be an injective sequence in $S$ with $s_n\to y\in (\G\cup X, \tau)$. If $(s_n)$ has a subsequence $(\g_n)$ in $\G$, then by (i) we have $y\in X$, and by (ii) we have $\g_n\nu\to \delta_y$ weak*. Since $S$ is closed in $\tau_{\nu}$, condition B implies $y\in S$.
If $(s_n)$ has a subsequence $(x_n)$ in $X$, then by (i), $y\in X$ and $x_n\to y$ in $X$. By condition A, $S\cap X$ is closed in $X$, which implies $y\in X$. Hence, it follows that $S$ is closed in $\tau$.
\end{proof}
A useful feature of the topology $\tau_\nu$ defined above is the following canonical extension property.
\begin{lemma}\label{lem:ext}
In the setting of Lemma~\ref{lem: measured compactification}, given $f\in C(X)$, the function $\Tilde{f}\colon \Gamma\cup X \to \CC$ defined by $\Tilde{f}|_{X}=f$ and $\Tilde{f}|_{\G}=\Pcal_{\nu}f$ is a continuous extension of $f$.
\begin{proof}
Suppose $\ler{\gamma_n}$ is a sequence in $\Gamma$ and $\gamma_n\goes x\in X$. Then by the definition of $\tau_\nu$, we have $\gamma_n\nu \goes \delta_x$ and hence $\Tilde{f}\ler{\gamma_n} = \int_X f\,d\gamma_n\nu \goes f(x) = \Tilde{f}\ler{x}$.
\end{proof}
\end{lemma}

Note however, in general,
$\ler{\Gamma\cup X , \tau_\nu}$ is not a compactification (not even necessarily compact). For instance, if $\nu$ is $\G$-invariant and non-atomic, %a measure which is preserved by $\Gamma$, 
then $\ler{\Gamma\cup X,\tau_\nu}$ is just the disjoint union of the topological spaces $\Gamma$ and $X$.

Our next goal is to characterize $\nu\in \Prob(X)$ for which $\tau_\nu$ is a $\G$-compactification.
Recall the following definition (see, for example, Section I-3 in~\cite{azencott2006espaces} or~Section V.2 in~\cite{Glas1976}). A measure $\nu$ on $X$ is said to be \emph{contractible} if for every $x\in X$, $\delta_x$ is in the weak* closure of the $\G$-orbit of $\nu$, or equivalently, if the Poisson transform $\Pcal_\nu$ is an isometry.

\begin{lemma}
\label{lem:criterion for compactness of taunu1}
$\G$ is dense in $(\Gamma\cup X,\tau_\nu)$ if and only if the Poisson transform $\Pcal_\nu\colon C\ler{X}\to \ell^\infty\ler{\Gamma}$ is an isometry. 
\end{lemma}
\begin{proof}
Assume first that $\Gamma$ is dense in $\Gamma\cup X$. Since any Poisson transform is always a contraction, we only need to show that
$\|f\|_\infty\le\|\Pcal_\nu f\|_\infty$ for every $f\in C\ler{X}$. For this, let $f\in C\ler{X}$ and choose $x_0\in X$ such that $\|f\|_\infty=f(x_0)$. By the assumption, 
there exists a sequence $(\gamma_n)$ in $\G$ such that $\gamma_n \goesn  x_0$. This means, by the definition of $\tau_\nu$, that $\gamma_n\nu\goesn\delta_{x_0}$, and therefore,
\[
\Pcal_\nu f \ler{\gamma_n} = \int_X f\,d\gamma_n \nu \goesn \int_Xf\,d \delta_{x_0}= f\ler{x_0}=\|f\|_\infty ,
\]
which implies $\|f\|_\infty\le\|\Pcal_\nu f\|_\infty$.

Conversely, assume that $\Pcal_\nu$ is an isometry, and for the sake of contradiction, suppose that
$\Gamma$ is not dense in $\G\cup X$. %By the definition of the topology $\tau_\nu$, it follows that there exists 
Let $x_0\in X$ be such that $x_0\notin \overline{\G}^{\, \tau_\nu}$ and choose $f\in C(X)$ such that $0\le f(x) \lneq f(x_0)$ for all $x\in X\setminus \{x_0\}$. 
Since $\Pcal_\nu$ is an isometry, %$\|\Pcal_\nu(f)\|_\infty = \|f\|_\infty$, and so 
there exists a sequence $(\gamma_n)$ in $\G$ such that $\Pcal_\nu f(\gamma_n) =\int_X f\,d\gamma_n \nu \goesn  \|f\|_\infty= f(x_0)$. Since $X$ is compact and metrizable, it is in particular separable, so $\Prob\ler{X}$ is not only compact but also metrizable, and we may assume without loss of generality that $(\gamma_n \nu)$ converges to some $\nu_0\in \Prob\ler{X}$. Hence $\int_X f\,d\nu_0=f(x_0)$, which by the choice of $f$ implies that $\nu_0 = \delta_{x_0}$. That is, $\gamma_n \nu\goesn \delta_{x_0}$, which means $\gamma_n\goesn {x_0}$ in $\tau_\nu$, a contradiction. This completes the proof.
\end{proof}

\begin{lemma}
\label{lem:criterion for compactness of taunu2}
Denote by $\pi\colon \ell^\infty\ler{\Gamma}\to \ell^\infty\ler{\Gamma}/c_0\ler{\Gamma}$ the canonical quotient map. 
The space $\Gamma\cup X$ with the topology $\tau_\nu$ is compact if and only if 
the composition $\f := \pi \circ \Pcal_\nu :C\ler{X} \to \ell^\infty\ler{\Gamma}/c_0\ler{\Gamma}$ is a $*$-homomorphism. 
\end{lemma}
\begin{proof}
$\implies$ Assume that $\Gamma \cup X$ is compact; and we will show $\f$ is a $*$-homomorphism. Since both the maps $\pi$ and $\Pcal_\nu$ are $*$-linear, the same holds for $\f$, so it only remains to show that $\f$ is multiplicative.
For this, let $f,g \in C\ler{X}$ and towards contradiction suppose that $\f \ler{fg} \ne \f\ler{f} \f \ler{g}$, that is, $\big(\Pcal_\nu fg - \Pcal_\nu f\, \Pcal_\nu g\big)\notin c_0(\G)$.
Thus,
there exists some $\epsi>0$ and an injective sequence $(\gamma_n)$, such that
\[
\abs{\int_X fg \, d\gamma_n\nu - \int_X f \, d\gamma_n\nu\int_X g \, d\gamma_n\nu}> \epsi
\]
for all $n\in \N$. By compactness of $\G\cup X$, we may assume $\gamma_n\goesn x_0$ for some $x_0\in X$, that is, $\gamma_n\nu\goesn \delta_{x_0}$ weak*. Then
\[
\lim_n\int_X fg \, d\gamma_n\nu = fg(x_0) = f(x_0)g(x_0) = \lim_n \int_X f \, d\gamma_n\nu\int_X g \, d\gamma_n\nu,
\]
a contradiction.\\

\noindent
    $\impliedby$ Conversely, assume $\f$ is multiplicative. 
    Since $X$ is compact, it is enough to show that every sequence in $\Gamma$ has a convergent subsequence. 
    Let $(\gamma_n)$  be an injective sequence in $\G$ and
for every $f,g\in C\ler{X}$, we have
\[
0=\lim_n\left(P_\nu fg - P_\nu f P_\nu g\right)(\gamma_n)=\lim_n\left(\int_X fg \, d\gamma_n\nu - \int_X f \, d\gamma_n\nu\int_X g \, d\gamma_n\nu\right) .
\]
Since $\Prob\ler{X}$ is compact and metrizable, $(\gamma_n)$ has a subsequence $(\gamma_{n_k})$ such that $\gamma_{n_k}\nu \underset{k\to\infty}{\goes}  \nu_0$ for some $\nu_0 \in \Prob\ler{X}$. 
It follows that
\[
\int_X fg \, d\nu_0 - \int_X f \, d\nu_0\int_X g \, d\nu_0 = 0 
\]
for all $f,g\in C\ler{X}$, which implies $\nu_0 = \delta_{x_0}$ for some $x_0\in X$.
 Hence, by the definition of $\tau_\nu$, $\gamma_{n_k}\goes x_0\in X$, which completes the proof.
\end{proof}

Combining Lemmas~\ref{lem:criterion for compactness of taunu1} and \ref{lem:criterion for compactness of taunu2}, we get the following characterization of measures $\nu\in \Prob(X)$ for which $\tau_\nu$ is a $\G$-compactification.

\begin{theorem}
 \label{criterion for compactness of taunu}
Let $X$ be a compact metrizable $\Gamma$-space, and let $\nu\in\pr(X)$. Then $\ler{\Gamma\cup X , \tau_\nu}$ is a $\Gamma$-compactification if and only if the following two conditions hold:
\begin{enumerate}
\item[\rm (i)]
the Poisson transform $\Pcal_\nu\colon C\ler{X}\to \ell^\infty\ler{\Gamma}$ is an isometry, and %with respect to the supremum norm in both sides, and
\item[\rm (ii)]
the composition $\pi \circ \Pcal_\nu :C\ler{X} \to \ell^\infty\ler{\Gamma}/c_0\ler{\Gamma}$ is a $*$-homomorphism, where $\pi\colon \ell^\infty\ler{\Gamma}\to \ell^\infty\ler{\Gamma}/c_0\ler{\Gamma}$ is the canonical quotient map. 
\end{enumerate}
In this case, we say $\ler{\Gamma\cup X , \tau_\nu}$ is an orbital compactification (defined by $\nu$).
\end{theorem}

The above two conditions in the statement of Theorem~\ref{criterion for compactness of taunu} obviously hold for Dirac measures.

\begin{corollary}
Let $X$ be a compact metrizable $\Gamma$-space, $x_0\in X$ and let $\nu=\delta_{x_0}$. Then $\ler{\Gamma\cup X , \tau_\nu}$ is a $\Gamma$-compactification, which we call the point-orbital compactification (defined by the point $x_0$).
\end{corollary}

We next give a characterization of orbital compactifications in terms of their corresponding function algebras; recall from Proposition~\ref{prop: every compactification is a spec of a subalgebra} that given a $\G$-compactification $\ler{\Gamma\cup X, \tau}$, we define $\Acal_\tau := \{f|_\Gamma : f\in C\ler{\Gamma\cup X , \tau}\}\subset \ell^\infty(\G)$.
\begin{proposition}\label{strucuture thm for measure comp.}
Let $\ler{\Gamma\cup X, \tau}$ be a $\G$-compactification. Then $\tau=\tau_\nu$ is orbital for some $\nu\in\pr(X)$ if and only if $\Acal_\tau = {\rm Im}\ler{\Pcal_\nu}\oplus c_0\ler{\Gamma}$.
\end{proposition}

\begin{proof}
Assume $\tau=\tau_\nu$ for some $\nu\in\pr(X)$. %We First show $\Image\ler{\Pcal_{\delta_{x_0}}} \cap c_0\ler{\Gamma}=\{0\}$. 
Suppose $f\in C(X)$ and $\Pcal_{\nu}f \in c_0\ler{\Gamma}$. Let $\Tilde{f}\colon \Gamma\cup X \to \CC$ be the continuous extension as in Lemma~\ref{lem:ext}. Since $\Tilde{f}|_\Gamma = \Pcal_{\nu}\ler{f}\in c_0\ler{\Gamma}$ it follows by Proposition \ref{a function vanishes on X iff in c0} that $f=\Tilde{f}|_X = 0$.
We conclude ${\rm Im}\ler{\Pcal_\nu}\cap c_0\ler{\Gamma} = \{0\}$.

Next, let $f\in C(X)$ and $\phi\in c_0(\G)$, we show $\Pcal_\nu (f)+ \phi \in \Acal_\tau$.
Let $\tilde f\in C(\G\cup X)$ be the extension of $f$ as in Lemma~\ref{lem:ext}, and also extend $\phi$ to the function $\tilde \phi\in C(\G\cup X)$ by defining $\tilde \phi|_X = 0$.
Then, $\Pcal_\nu (f)+ \phi= (\tilde f+\tilde \phi)|_\G \in \Acal_\tau$.

Conversely,  
let $\tilde f\in C\ler{\Gamma \cup X}$ and we show 
$\tilde f|_\Gamma- \Pcal_\nu (\tilde f|_X) \in c_0 \ler{\Gamma}$,
which then implies $\tilde f\in {\rm Im}\ler{\Pcal_\nu}\oplus c_0\ler{\Gamma}$.
For the sake of contradiction, assume otherwise, that is, 
there exist $\epsi>0$ and an injective sequence $(\gamma_n) \subset \Gamma$, such that
\[
\abs{\tilde f\ler{\gamma_n} - \Pcal_\nu \ler{\tilde f|_X}\ler{\gamma_n}}>\epsi
\]
for all $n\in\N$.
By compactness of $(\Gamma\cup X,\tau_\nu)$, we may further assume that $\gamma_{n}\to  x \in X$, that is, $\gamma_{n}\nu \to \delta_x$ weak*. 
Then, by continuity, $\tilde f\ler{\gamma_{n}} \to \tilde f\ler{x}$, and 
\[ 
\Pcal_\nu \ler{\tilde f|_X}\ler{\gamma_{n}} =\int_X \tilde f|_X \,d \gamma_{n}\nu \goesn  \tilde f\ler{x} .
\]
Combining the above, we get a contradiction. This completes the proof of the forward implication of the first assertion.

For the converse implication, assume $\ler{\Gamma\cup X, \tau}$ is a $\G$-compactification such that $\Acal_\tau = {\rm Im}\ler{\Pcal_\nu}\oplus c_0\ler{\Gamma}$ for some $\nu\in\pr(X)$.
Let $\Tilde{f}\in C(\G\cup X)$. By the assumption, $\Tilde{f}|_\G=\Pcal_{\nu}(\Tilde{f}_1)+\phi$, where $\Tilde{f}_1\in C(\G\cup X)$ and $\phi\in c_0(\G)$. Extend $\phi$ to $\Tilde{\phi}\in C(\G\cup X)$ by setting $\Tilde{\phi}|_X=0$. Then 
$\Tilde{f}|_\G=\Tilde{f}_1|_\G+\Tilde{\phi}|_\G$, and so by the density of $\G$, we get $\Tilde{f}=\Tilde{f}_1+\Tilde{\phi}$.

Let $(\gamma_n)$ be an injective sequence in $\G$. Then $\Tilde{\phi}(\gamma_n)\goes 0$, and therefore if $(\gamma_n)$ converges in $\ler{\Gamma\cup X, \tau}$, then
\[
\lim_n \Tilde{f}(\gamma_n)= \lim_n \Tilde{f}_1(\gamma_n) =\lim_n \Pcal_{\nu}(\Tilde{f}_1)(\gamma_n) =
\lim_n \int_X f\,d\gamma_n\nu.\]
Since this holds for every $f\in C(\G\cup X)$, it follows $\gamma_n\goes x\in X$ if and only if $\gamma_n\nu\goes \delta_x$ weak*. This implies $\tau=\tau_\nu$.
\end{proof}

\subsection{Examples}
We proceed with several examples of $\G$-compactifications. First, we provide two elementary and natural  examples of orbital $\G$-compactifications, one of which is point-orbital, and the other not. %an example of a Dirac compactification. 

\begin{example} \label{example: Dirac comp. for ll(Z)}
Consider the lamplighter group on $\ZZ$,
\[
\Gamma = \ZZ\ltimes \oplus_\ZZ (\ZZ/2\ZZ) \cong (\ZZ/2\ZZ) \wr \ZZ \cong \innprod{a,t \mid a^2, \ler{at^nat^{-n}}^2 \text{ for all } n\in\ZZ},
\]
acting on $X:=\prod_{\ZZ} \ZZ/2\ZZ$, the space of infinite lamp configurations, equipped with the standard product topology, that makes it a compact space. 
The action is defined as follows. For an element $\ler{p,C}\in\Gamma$ (where $p$ is the position of the lamplighter and $C$ is a finite lamps configuration), and a point $\Bar{C}\in X$ (a finite or infinite configuration), the the resulted element $\ler{p,C}.\Bar{C}$ is defined by translating $\Bar{C}$ by $p$ and then summing (xoring, if one prefers) with $C$. 

Since every element $\gamma = \ler{p,C} \in \Gamma$ possesses a finite configuration, we can use that to construct a $\Gamma$-compactification by declaring
\[
\gamma_n = \ler{p_n,C_n}\goesn  \Bar{C} \iff
C_n \goesn  \Bar{C} \text{ (in X)}.
\]
Now, one can easily verify that this is indeed a $\Gamma$-compactification, and it is the point-orbital compactification through the empty configuration $\ler{...,0,0,0,0,...}\in X$. One can easily check manually that the compactification is indeed point-orbital, or by using Theorem \ref{non-Dirac iff Brouwer} (consider the retraction that fixes the boundary and ``forgets" the pointer on the group, i.e., $\ler{p,C}\mapsto C$).
\end{example}

\begin{example} \label{gromov not Dirac}
Let $\FF_2$ be the free group on two generators, let $\partial\FF_2$ be its Gromov boundary, and denote by $\tau_G$ the usual topology on $\FF_2 \cup \partial\FF_2$ of the Gromov-compactification (see \cite{Gro1987}). 
Then $\ler{\FF_2\cup\partial \FF_2 , \tau_G}$ is not a point-orbital compactification.
To see this, suppose on the contrary that $\tau_G = \tau_{x_0}$ is a point-orbital compactification for some $x_0 \in \partial \FF_2$. We split the proof into cases and show a contradiction in each of them.
First, consider the case that $x_0$ has eventually only one letter from $\lers{a^{\pm 1},b^{\pm 1}}$, that is $x_0 = w_0\alpha^\infty$, for some finite word $w_0\in \FF_2$ and letter $\alpha\in\lers{a^{\pm 1},b^{\pm 1}}$. Assume without loss of generality that $x_0 = w_0 a^\infty$. In this case, the contradiction comes from the sequence $(w_n)$ in $\FF_2$ given by $w_n := w_0a^{-n}$ for every $n$. It is clear that we should have that $w_n \to  w_0a^{-\infty}$ in $\tau_G$. But $w_0a^{-n}a^{\infty}= w_0a^\infty$, hence $w_nx_0 = w_0a^{-n}a^{\infty} \to  w_0a^\infty \ne w_0a^{-\infty}$.

Now, assume that $\tau_G = \tau_{x_0}$, a point-orbital compactification such that $x_0$ does not eventually contain only one letter. Assume, without loss of generality, that the letter $a$ appears infinitely many times in $x_0$. Consider the sequence $(w_n)$ in $\FF_2$ which is defined by taking the inverse of the prefix of $x_0$ that goes up to right before the $n$-th power of $a$ in $x_0$. For instance, if
\[
x_0 = a^2b^{-1}a^{-3} b^5 a b a^7...,
\]
then
\[
\begin{split}
 w_1 &= e \ \text{(the empty word)} \\
w_2 &= ba^{-2} \\
w_3 &= b^{-5} a^3 b a^{-2} \\
w_4 &= b^{-1} a^{-1} b^{-5} a^3 b a^{-2}\\
&\ \ \vdots
\end{split}
\]
For every letter $\alpha\in\lers{a^{\pm 1}, b^{\pm 1}}$, denote by $\lesq{\alpha} \subset \partial \FF_2$ the closed set of infinite words which start with $\alpha$. By construction, the first letter of $w_n$ is $b$ or $b^{-1}$ for every $n\geq2$. Hence, every partial limit of $(w_n)$ must lie in $\lesq{b} \cup \lesq{b^{-1}}$. However, $w_n  x_0 \in \lesq{a} \cup \lesq{a^{-1}}$ and hence, every partial limit of $(w_nx_0)$ is in $\lesq{a} \cup \lesq{a^{-1}}$. Nevertheless, we supposed that $\tau_G = \tau_{x_0}$, and therefore a subsequence $\ler{w_{n_k}}_k$ would converge to some $x\in\partial \FF_2$ if and only if $w_{n_k}x_0 \underset{k\to\infty}{\goes}  x$, in contradiction with the fact that $(w_n)$ and $\ler{w_nx_0}$ do not share any partial limit.
This proves the claim.
\end{example}

In Theorem \ref{gromov not Dirac} we extend the above ``hands-on'' argument to a more general framework to verify that the Gromov compactifications of hyperbolic groups are not point-orbital. 

\subsubsection*{Non-orbital $\Gamma$-compactifications}
At this point, it is important to emphasize that not every $\Gamma$-compactification is orbital. %, as reflected in the following examples. 
\begin{example} \label{example: two points compactification, non-trivial}
Consider $\ZZ$ acting on the space $X=\lers{a,b}$ of two points transitively ($1.a=b$). One can easily check that the following conditions give rise to a unique $\Z$-compactification, $\tau$,
\begin{itemize}
\item $2,4,6,8,...\goes  a$
\item $-2,-4,-6,-8,... \goes  b$
\item $1,3,5,7,... \goes  b$
\item $-1,-3,-5,-7,...\goes  a$
\end{itemize}
Then $\tau$ is not point-orbital, since if, for instance, $\tau=\tau_a$, we know that the even numbers act trivially, but then
\[
\lim_{n\to\infty} -2n = b \ne a = \lim_{n\to\infty} -2n.a.
\]
Hence, the compactification here is not orbital as well.
\end{example} 

\begin{example} \label{example: infinite dihedral grp}
Consider the infinite dihedral group $D_\infty= \innprod{\rho,\sigma}$ acting on $\Z$ where $\rho$ is the right shift and $\sigma$ is the reflection. Consider its action on the two-point space $\lers{a,b}$, in which $\rho$ acts trivially and $\sigma$ permutes $a$ and $b$. To obtain a $D_\infty$-compactification here, we define a topology $\tau$ as follows. Consider the action of $D_\infty$ on $\ZZ$ described above. We say that a sequence $\ler{\gamma_n}$ in $D_\infty$ converges to $a$ if $\gamma_n.0\goes\infty$ and converges to $b$ if  $\gamma_n.0\goes  -\infty$. If $\tau=\tau_a$, then since $\rho$ fixes both $a$ and $b$, we would get that $\rho^{-n}$ converges to $a$. Similarly, if $\tau=\tau_b$, then $\rho^n$ would converge to $b$. In both cases, we get a contradiction with the definition of $\tau$. By the same argument as in Example \ref{example: two points compactification, non-trivial}, $\tau$ is also not an orbital compactification.  
\end{example}

\section{No-retraction theorems}

Now we restate and prove Theorem \ref{non-Dirac iff Brouwer} from the introduction.

\begin{theorem} \label{one pt iff extension}
Let $X$ be a metrizable $\G$-space and $(\Gamma \cup X,\tau)$ be a $\Gamma$-compactification. Then there exists a $\Gamma$-equivariant retraction $\f\colon \Gamma \cup X \to X$ if and only if $\tau = \tau_{x_0}$ is point-orbital for some $x_0 \in X$.
In this case, we have $x_0=\f\ler{e}$. \end{theorem}

\begin{proof}
Assume $\tau = \tau_{x_0}$ for some $x_0\in X$. Define $\varphi: \G\cup X\to X$ by $\varphi(\g)= \g x_0$ for all $\g\in\G$, and $\varphi(x)=x$ for all $x\in X$. Then $\varphi$ is a $\G$-equivariant map with $\f|_X=id_X$. We show it is continuous. For this, let $(\g_n)$ be a sequence in $\G$ such that $\g_n\to x\in X$. By the definition of the topology $\tau_{x_0}$, this means $\g_nx_0\to x$, and so $\varphi(\g_n)= \g_n x_0\to x = \varphi(x)$. Hence, $\varphi$ is continuous.

Conversely, assume $(\Gamma \cup X,\tau)$ is a $\Gamma$-compactification such that there exists a $\Gamma$-equivariant retraction $\f\colon \Gamma \cup X \to X$.
We show $C\ler{\Gamma\cup X}|_\G = c_0 \ler{\Gamma} \oplus \Image\big(\Pcal_{\delta_{\f\ler{e}}}\big)$, and then apply Proposition~\ref{strucuture thm for measure comp.} to conclude the result.

Let $\Tilde{f}\in C\ler{\Gamma\cup X}$, and denote $f=\Tilde{f}|_X$. Then, for every $\g\in \G$ we have
\[
\Pcal_{\delta_{\f\ler{e}}}f\ler{\gamma} =
f\ler{\gamma\f\ler{e}}=
 f\circ \f \ler\gamma ,
\]
and so $\Pcal_{\delta_{\f\ler{e}}} (\Tilde{f}|_X)= (\Tilde{f}|_X\circ \f)|_\G$.

Next, we show $(\tilde{f}-(f\circ\f))|_\G \in c_0\ler{\Gamma}$. Suppose for contradiction that this is not so, %suppose that $f-\psi\ler{f|_X} \notin c_0\ler{\Gamma}$. T
that is, that there exists an $\epsi>0$ and an injective sequence $(\gamma_n)_{n}$ in $\G$ such that
\begin{equation} \label{Jasmine}
\abs{\tilde{f}\ler{\gamma_n}-\tilde{f}\ler{\f(\gamma_n})} \geq \epsi
\end{equation}
for all $n$. 
By compactness, we may also assume without loss of generality that $\gamma_n\goes x\in X$. Since $\f$ is a retraction, we get
\[
0=\tilde{f}\ler{x}-\psi(\tilde{f}|_X)\ler{x} = \lim_{n\to\infty} f\ler{\gamma_n}-\psi(\tilde{f}|_X)\ler{\gamma_n},
\]
in contradiction.
Thus,
\[
\tilde{f}|_\G = \big(\tilde{f}-(f\circ\f)\big)|_\G + (f\circ\f)|_\G = \big(\tilde{f}-(f\circ\f)\big)|_\G + \Pcal_{\delta_{\f\ler{e}}}f\in c_0(\G) + {\rm Im}\left(\Pcal_{\delta_{\f\ler{e}}}\right),
\]
which shows $\Acal_\tau\subset c_0(\G) + {\rm Im}\left(\Pcal_{\delta_{\f\ler{e}}}\right)$.

Conversely, let $f\in C(X)$ and $\phi\in c_0(\G)$. Let $\Tilde{f}= f\circ\f\in C(\G\cup X)$ and let $\tilde{\phi}\in C(\G\cup X)$ be the continuous extension of $\phi$ by $\Tilde{\phi}|_X=0$. Then,
\[
f+\phi = (\Tilde{f}+\tilde{\phi})|_\G \in \Acal_\tau .
\]
Hence, $\Acal_\tau= c_0(\G) + {\rm Im}(\Pcal_{\delta_{\f\ler{e}}})$.

It remains to show $c_0(\G) \cap {\rm Im}\left(\Pcal_{\delta_{\f\ler{e}}}\right) = \{0\}$. For this, let $f\in C(X)$ be such that $\Pcal_{\delta_{\f\ler{e}}}f\in c_0(\G)$. Then, as shown above, it follows $f\circ\f= \Pcal_{\delta_{\f\ler{e}}}f\in c_0(\G)$. Since $f\in c_0(\G)\in C(\G\cup X)$ and $\G$ is dense in $\G\cup X$, and it follows $(f\circ\f)|_X= 0$, and since $\f$ is a retraction, $f=(f\circ\f)|_X= 0$.
This completes the proof.
\end{proof}
In fact, all orbital group compactifications satisfy a ``quasi'' version of equivariant no-retraction theorem as follows; the proof is similar to Theorem~\ref{one pt iff extension}, which we will omit.

\begin{theorem} %\label{one pt iff extension}
A $\Gamma$-compactification $(\Gamma \cup X,\tau)$ is orbital if and only if there exists a continuous $\Gamma$-equivariant map $\f: \Gamma \cup X \to \pr(X)$ such that $\f(x)=\delta_x$ for all $x \in X$.
\end{theorem}

With the complete characterization given by Theorem~\ref{one pt iff extension} in hand, we can now prove the equivariant analogue of Brouwer's No-Retraction theorem for several classes of group compactifications.
We begin with the important case of Gromov compactification of hyperbolic groups.

\begin{theorem}\label{non-retract thm in Gromov comp.}
Let $\Gamma$ be a finitely generated hyperbolic group and consider its Gromov-compactification $\ler{\Gamma\cup \partial \Gamma,\tau_G}$. Then $\tau_G$ is an orbital, but not point-orbital compactification.
\end{theorem}

\begin{proof}
Let $\nu\in\pr(\partial \G)$ be non-atomic. By \cite[Lemma 2.2 \& Proposition 7.5]{Kai00}, if $(\g_n)$ is a sequence in $\G$ such that $\g_n\to x\in \partial \G$ in $\tau_G$, then $\g_n\nu\to \delta_x$ weak*. Since both $\tau_G$ and $\tau_\nu$ are compact topologies, it follows that $\tau_G=\tau_\nu$, and hence $\tau_G$ is orbital.

To prove that $\tau_G$ is not point-orbital, assume for the sake of contradiction, that $\tau_G= \tau_{x_0}$ for some $x_0\in X$.
Let $\omega\colon\Z_{\ge 0}\to \G$ be a geodesic with $\g_0:=\omega(0)=e$ and such that $\g_n:=\omega(n)\to x_0$. For every $m< n\in \N$, since $e$ lies in a geodesic segment $[\g_m^{-1}\,,\, \g_m^{-1}\g_n]$, the Gromov product $\langle\,\g_m^{-1}\,,\, \g_m^{-1}\g_n\,\rangle_e = 0$, and therefore $\langle\,\g_m^{-1}\,,\, \g_m^{-1}x_0\,\rangle_e = 0$ for all $m\in\N$.
Let $\ler{\g_{m_j}^{-1}}_j$ be a subsequence that converges to infinity, say $\g_{m_j}^{-1}\to z\in \partial \G$. Then we also have $\g_{m_j}^{-1}x_0\to z$, and so $\exists N\in \N$ such that $\langle\,\g_{m_i}^{-1}\,,\, \g_{m_j}^{-1}\,\rangle_e>(2\delta+1)$ and $\langle\,z\,,\, \g_{m_j}^{-1}x_0\,\rangle_e> (2\delta+1)$
for all $i, j\ge N$, where $\delta$ is the hyperbolicity constant of $\G$. Fix $j\ge N$, and the second inequality further implies that $\exists M\ge N$ such that $\langle\,\g_{m_\ell}^{-1}\,,\, \g_{m_j}^{-1}\g_n\,\rangle_e\ge (2\delta+1)$ for all $\ell, n\ge M$.
However, then for $M<m_i<n$,
by $\delta$-hyperbolicity we have
\[
0 = \langle\,\g_{m_i}^{-1}\,,\, \g_{m_i}^{-1}\g_n\,\rangle_e \ge \min\{\langle\,\g_{m_i}^{-1}\,,\, \g_{m_j}^{-1}\,\rangle_e \,,\, \langle\,\g_{m_j}^{-1}\,,\, \g_{m_i}^{-1}\g_n\,\rangle_e\} - 2\delta
\ge 1 ,
\]
which is a contradiction.
    \end{proof}

\begin{corollary}
Let $\Gamma$ be a finitely generated hyperbolic group and consider its Gromov-compactification $\ler{\Gamma\cup \partial \Gamma,\tau_G}$. Then there exists no $\G$-equivariant retraction from the compactification $\G\cup \partial \Gamma$ onto its boundary $\partial \Gamma$.
\end{corollary}

\subsection{%No-retraction theorems for 
Projective compactifications}

Recall from \cite{Kai00} that a $\G$-compactification $(\G\cup X, \tau)$ is said to be \emph{projective} if the right action $\G\curvearrowleft\G$ extends to the trivial action of $\G$ on $X$, that is, if $(\g_n)$ is a sequence in $\G$ and $\g_n\to x\in X$, then $\g_n\g\to x\in X$ for every $\g\in\G$.

Gromov compactifications of hyperbolic groups are examples of projective compactifications. Although it is not completely evident in the proof of Theorem~\ref{non-retract thm in Gromov comp.}, we believe that projectivity is the main obstacle for point-orbitality.
In particular, we conjecture that the equivariant no-retraction theorem holds for all projective group compactifications.

The remainder of this section is devoted to some facts concerning this class of compactifications.

\medskip

Recall that an action $\Gamma\acts X$ of $\G$ on a compact metric space $(X, d)$ is called \textbf{distal} if for every pair of distinct points $x,y\in X$, $\inf\{d(\g x, \g y) : \g\in \G\}>0$. 

\begin{theorem}\label{thm:distal}
A projective group compactification with non-trivial distal boundary is not point-orbital.
\end{theorem}

\begin{proof}
Suppose, for the sake of contradiction, that there is a non-trivial distal minimal compact $\G$-space $X$ and a point $x_0\in X$ such that $\ler{\Gamma\cup X,\tau_{x_0}}$ is projective.
Let $\ler{\gamma_n}$ be a sequence in $\Gamma$ such that $\g_nx_0\to x\in X$. Then, by projectivity of $\tau_{x_0}$, we also get $\g_n\g x_0\to x$, for all $\g\in \G$. Since $\G\act X$ is distal, it follows $\g x_0= x_0$ for all $\g\in \G$. By minimality, it follows that $X$ is trivial, a contradiction.
\end{proof}

\begin{corollary}
Let $\ler{\Gamma\cup X,\tau}$ be a projective $\Gamma$-compactification such that $X$ is a non-trivial distal $\Gamma$-space. Then there is no $\G$-equivariant retraction from $\Gamma\cup X$ onto $X$.
\end{corollary}

In fact, projectivity, together with the point-orbitality of a group compactification, imposes topological restrictions on its boundary, as shown below.

\begin{proposition}\label{con-or-inf-dis}
The boundary of any projective point-orbital $\G$-compactification is either connected or has infinitely many connected components.
\end{proposition}

\begin{proof}
For the sake of contradiction, assume that $(\G\cup X, \tau_{x_0})$ is projective for some $x_0\in X$, and $X$ is disconnected, with finitely many connected components. Thus, there exists a proper clopen connected component $K\subset X$ with $x_0\in K$. Define $A:=\{\g\in \G \,:\, \g x_0\in K\}$. 
Let $g\in \G$. Since $K$ is clopen, it follows that if $Ag\setminus A$ is infinite, then there is a sequence $(\g_n)$ in $A$ that converges in $K$, but $(\g_ng)$ converges in $X\setminus K$, which contradicts the projectivity. Thus, $A$ is right-invariant up to finite sets, that is, for every $g\in \Gamma$, the set $Ag \setminus A$ is finite. Moreover, for every $\g\in A$, since $x_0\in K\cap \g^{-1}K$, it follows $\g^{-1}K\subset K$, and so $\g^{-1}\in A$. Therefore, $A$ is symmetric, hence also is left-invariant up to finite sets. By minimality, the set $A$ is infinite, and there exists an injective sequence $(\g_n)$ in $A$ such that $\g_n x_0\to x_0$, so $\g_n\to x_0$ in $\tau_{x_0}$. 
Let $g\in A^c$. Then $g\g_n\to g x_0$, which implies that for infinitely many $n\in\mathbb{N}$, $g\g_n\notin A$. This contradicts left-invarinace of $A$ by finite sets, and completes the proof.
\end{proof}

We have further topological restriction on the boundary in the case of 1-ended groups.

\begin{proposition}
Let $\G$ be a finitely generated $1$-ended group. Then the boundary of any projective point-orbital $\G$-compactification must be connected.
\end{proposition}

\begin{proof}
Suppose $(\G\cup X, \tau_{x_0})$ is projective for some $x_0\in X$, where $X$ is disconnected. 
Then similarly as in the proof of Proposition~\ref{con-or-inf-dis}, $\G$ has a subset $A$ that is right-invariant up to finite sets, and both $A$ and $X\setminus A$, its complement in $X$, are infinite.
The existence of such subset $A\subset \G$ implies that $\G$ has more than one end. 
\end{proof}

We observe below that, as one would expect, projective group compactifications of 2-ended groups have finite boundary.

\begin{proposition}\label{2-end}
Suppose $\Gamma$ contains an infinite subgroup whose centralizer has finite index in $\G$ (e.g. $\Gamma$ is $2$-ended), then every projective $\G$-compactification has finite boundary.
Consequently, the equivariant no-retraction theorem holds for any such compactification with non-trivial boundary.
\end{proposition}

\begin{proof}
Let $H\leq\Gamma$ be an infinite subgroup and assume that its centralizer, $C_\Gamma\ler{H}$, is of finite index in $\Gamma$. Suppose that $\ler{\Gamma\cup X , \tau}$ is a projective $\G$-compactification. Let $x\in X$ be a limit point of $H$, that is, $h_n\to x$ for some sequence $(h_n)$ in $H$.
For every $\g\in C_\Gamma\ler{H}$, by projectivity,
\[
x = \lim_{n\to\infty}h_n = \lim_{n\to\infty} {h_n} \g  =\lim_{n\to\infty} \g {h_n} 
=\g \lim_{n\to\infty} {h_n} 
= \g x ,
\] 
which shows that $x$ is fixed by $C_\Gamma\ler{H}$. Since $\G\act X$ is minimal and $C_{\Gamma}(H)$ has finite index in $\G$, it follows that $X$ is finite.

In particular, since every finite $\G$-space is distal, it follows from Theorem~\ref{thm:distal} that if $X$ is non-trivial, then there is no $\G$-equivariant retraction from $\G\cup X$ onto $X$.
\end{proof}

    \printbibliography

   \end{document}